\documentclass[pdftex]{amsart} 

\usepackage{mathtools}

\usepackage{array,rotating,pdflscape,
  amsfonts,amssymb,amsmath,latexsym,upref,enumerate,commath,cool}
\usepackage[margin=2.0cm]{geometry}
\usepackage{verbatim} 
\usepackage[english]{babel}
\usepackage[latin1]{inputenc} 
\usepackage{color,tikz,pgf}

\usetikzlibrary{shapes, intersections, calc, patterns, positioning,
  matrix, arrows, fit, backgrounds, 
  decorations.pathmorphing, decorations.markings}

\newtheorem{lemma}[equation]{Lemma}
\newtheorem{prop}[equation]{Proposition}
\newtheorem{thm}[equation]{Theorem}
\newtheorem{cor}[equation]{Corollary}

\theoremstyle{definition}
\newtheorem{exmp}[equation]{Example}

\numberwithin{equation}{section}

\newcommand{\Z}{\mathbf{Z}}

\newcommand{\Q}{\mathbf{Q}}

\newcommand{\F}{\mathbf{F}}

\newcommand{\qt}[2]{#1 \backslash #2}

\newcommand{\SOodd}[2]{\operatorname{SO}_{#1}(#2)}
\newcommand{\SOeven}[3]{\operatorname{SO}^{#1}_{#2}(#3)}
\newcommand{\Spinodd}[2]{\operatorname{Spin}_{#1}(#2)}
\newcommand{\Spineven}[3]{\operatorname{Spin}^{#1}_{#2}(#3)}

\newcommand{\GL}[3]{\operatorname{GL}^{#1}_{#2}(\F_{#3})}

\newcommand{\SL}[3]{\operatorname{SL}^{#1}_{#2}(#3)}

\newcommand{\Lie}{\mathit{Lie}}
\newcommand{\Ob}[1]{\mathrm{Ob}(#1)}

\newcommand{\cat}[3]{\mathcal{#1}_{#2}^{#3}}  
  
\newcommand{\m}{morphism}
\newcommand{\pol}{polynomial}

\newcommand{\syl}[1]{Sylow $#1$-subgroup}
\newcommand{\we}{weighting}
\newcommand{\Mb}{M\"obius}
\newcommand{\Euc}{Euler characteristic}
\newcommand{\rchi}{\widetilde{\chi}}

\newcommand{\Syl}[2]{\operatorname{Syl}_{#1}(#2)}

\newcommand{\mynote}[1]{\noindent{\textcolor{red}{\textbf{[#1]}}}}

\newcommand{\wh}[1]{\widehat{#1}}

\title{ Euler characteristics and $p$-singular elements in finite groups}

\author{Jesper M.~M\o ller}
\address{Institut for Matematiske Fag\\
  Universitetsparken 5\\
  DK--2100 K\o benhavn}
\email{moller@math.ku.dk}
\urladdr{htpp://www.math.ku.dk/~moller}

\thanks{Supported by the Danish National Research Foundation through
  the Centre for Symmetry and Deformation (DNRF92)}

\usepackage[bookmarks=true,bookmarksopen=false]{hyperref}
\hypersetup{
   pdftitle = {},
   pdfauthor = {Jesper Michael MÃÂÃÂ¸ller},
   pdfpagemode = {UseOutlines},
   pdfstartview = {FitH},
   pdfborder = {0 0 0},
   backref = {true},
   colorlinks = {true},
   urlcolor = {blue},
   citecolor = {blue},
   linkcolor = {blue},
   pdftoolbar = {true},
}


\begin{document}
\date{\today}
\begin{abstract}
  We use the Euler characteristic of the orbit category of a finite
  group to establish equivalences between theorems of Frobenius and
  K.S. Brown and between theorems of Steinberg and L. Solomon.
\end{abstract}
\maketitle

\section{Introduction}
\label{sec:introduction}

Let $G$ be a finite group with unit element $e$, $p$ a prime
number and $|G|_p$ the $p$-part of the group order.  An
element of $G$ is {\em $p$-singular\/} if its order is a
power of $p$ \cite[Definition 40.2, \S 82.1]{cr}.  Write
\begin{equation*}
  G_p = \{ g \in G \mid g^{|G|_p} = e\} = \bigcup \Syl pG
\end{equation*}
for the set of all $p$-singular elements in $G$.  In other words, $G_p$
is the solution set in $G$ to the equation $X^{|G|_p}=1$ or the union of
all the Sylow $p$-subgroups of $G$.\@ A theorem of Frobenius from
$1907$, or even earlier,
\begin{equation*}
  n \mid |G| \implies n \mid |\{ g \in G \mid g^n=e \}|
\end{equation*}
contains as a
special case a basic fact about the number, $|G_p|$, of $p$-singular
group elements \cite{frobenius:1907, isaacsrobinson}
\cite[Corollary~41.11]{cr} \cite[11.2, Corollary~2]{serre77}.

\begin{thm}[Frobenius $1907$]
\label{thm:frobenius}
  $|G|_p \mid |G_p|$
\end{thm}

The number of $p$-singular elements is known for the symmetric groups
and for the finite groups of Lie type in defining characteristic $p$:
\begin{itemize}
\item The exponential generating function for the  
  number of $p$-singular permutations in the symmetric groups $\Sigma_n$
is \cite[Example
  5.2.10]{stanley99}
    \begin{equation*}
      \sum_{n=1}^\infty  |(\Sigma_n)_p| \frac{x^n}{n!} =
      \exp ( x + \frac{x^p}{p} + \frac{x^{p^2}}{p^2}+ \cdots + 
      \frac{x^{p^m}}{p^m} + \cdots )
    \end{equation*}
  \item $|K_p| = |K|_p^2$   for a
    finite group $K$ of Lie type in defining characteristic $p$
    (Theorem~\ref{thm:steinberg}) 
\end{itemize}

The aim of this note is to relate Frobenius' theorem
(Theorem~\ref{thm:frobenius}) to a theorem of
K.S. Brown (Theorem~\ref{thm:brown}) and Steinberg's theorem
(Theorem~\ref{thm:steinberg}) to a theorem of L. Solomon
(Theorem~\ref{thm:solomon}). These two pairs of
theorems are linked by the \Euc\ of the orbit category discussed in 
Proposition~\ref{prop:GpEuc}.\eqref{item:GpEuc1}.

The following notation will be used in this note:

\begin{tabular}[h]{c|l}
  $G$ & a finite group \\
  $\cat SG{}$ & the poset of subgroups of $G$ ordered by inclusion, $H
  \leq K \iff H  \subseteq K$\\
  $\cat OG{}$ & the orbit category of subgroups of $G$, $\cat
                OG{}(H,K) = \{ g \in G \mid H^g \subseteq K\}/K$,
                $\cat OG{}(H) = N_G(H)/H$ \\
  $p$ & a prime number \\
  $O_p(G)$ & the biggest normal $p$-subgroup of $G$ \\
  $n_p$ &  the $p$-part of the natural number $n$ 
  \\
  $\cat C{}{}(a,b)$ & set of \m s from $a$ to $b$ in category $\cat
                      C{}{}$ \\
  $\cat C{}{}(a)$ & monoid $\cat C{}{}(a,a)$ of endo\m s of $a$ in
                    category $\cat C{}{}$ \\
  $q$ &  a prime power \\
  $\F_q$ &  the finite field with $q$ elements 
\end{tabular}

If $\cat C{G}{}$ is a category whose objects are all subgroups of
$G$, then $\cat C{G}{p}$,  $\cat C{G}{p+*}$,  $\cat C{G}{p+\mathrm{rad}}$ denotes
the full subcategory of $\cat C{G}{}$ on all $p$-subgroups, non-trivial
$p$-subgroups, $p$-radical $p$-subgroups, respectively. (A $p$-subgroup $H$
of $G$ is $p$-radical if $H=O_pN_G(H)$.)

\section{Using \Euc s to count $p$-singular elements}
\label{sec:using-euc-s}

We apply Tom
Leinster's theory of \Euc s of
finite  categories \cite{leinster08}  to the orbit
category $\cat OGp$.

Let $\zeta$ be a square matrix with rational coeffecients. A {\em
  \we\/} for $\zeta$ is a vector $k$ such that all coordinates of
$\zeta k$ equal $1$. A {\em co\we\/} is a \we\ for the transpose of
$\zeta$. The matrix $\zeta$ has \Euc\ if it admits both a \we\ and a
co\we , and the \Euc\ of $\zeta$, $\chi(\zeta)$, is then the coordinate
sum of a \we\ or a co\we\ \cite[Lemma~2.1,
Definition~2.2]{leinster08}. The \Euc\ of an invertible matrix is the
sum of the entries of the inverse.

Let $\cat C{}{}$ be a finite category. The $\zeta$-matrix of
$\cat C{}{}$ is the square matrix
$\zeta(\cat C{}{}) = \left( |\cat C{}{}(a,b)| \right)_{a,b \in
  \Ob{\cat C{}{}}}$ recording the cardinalities of all the \m\ sets in
$\cat C{}{}$. A \we\ or co\we\ for $\cat C{}{}$ is a \we\ or co\we\ for
$\zeta(\cat C{}{})$.  The \Euc\ of $\cat C{}{}$ is
$\chi(\cat C{}{}) = \chi(\zeta(\cat C{}{}))$ when $\cat C{}{}$ has a
\we\ and a co\we . The {\em reduced\/} \Euc\ of $\cat C{}{}$ is
$\rchi(\cat C{}{}) = \chi(\cat C{}{})-1$. The \Euc\ of any finite
category with an initial or terminal object is $1$.

The finite categories of this note all admit \we s and co\we s.
By {\em the\/} \we\ for e.g.\ $\cat OGp$ we mean the \we\ that is constant
on iso\m\ classes of objects \cite[p 3035]{gm:2012}. 

\begin{lemma}\label{lemma:Gp}
  The number of $p$-singular elements in $G$ is
  \begin{equation*}
    |G_p| =
    p^{-1} +  \sum_{1 \leq C\leq G}(1-p^{-1})|C|
  \end{equation*}
  where the sum is over all cyclic $p$-subgroups $C$ of $G$.
\end{lemma}
\begin{proof}
  Declare two $p$-singular elements to be equivalent of they generate
  the same cyclic subgroup. The set of equivalence classes is the set
  of cyclic $p$-subgroups $C$ of $G$.  The number of elements in the
  equivalence class $C$ is the number of generators of $C$: $p^{-1}+(1-p^{-1})|C|$ if
  $|C|=1$ and $(1-p^{-1})|C|$ if $|C| > 1$. 
\end{proof}

The \we s for the poset $\cat SGp$ and the category $\cat OGp$,
$k^K_{\cat S{}{}}=-\rchi(\cat S{\cat OG{}(K)}{p+*})$ and
$k^K_{\cat O{}{}}=-\frac{1}{|G|}\rchi(\cat S{\cat OG{}(K)}{p+*})|K|$
\cite[Theorem 1.3]{jmm_mwj:2010}, vanish off the $p$-radical subgroups
by Quillen's \cite[Proposition 2.4]{quillen78}. Thus the \we s for
$\cat SGp$, $\cat OGp$ restrict to \we s for the full subcategories
$\cat SG{p+\mathrm{rad}}$, $\cat OG{p+\mathrm{rad}}$ and
$\chi(\cat SG{p+\mathrm{rad}}) = \chi(\cat SGp) = 1$,
$\chi(\cat OG{p+\mathrm{rad}}) = \chi(\cat OGp)$ \cite[Lemma~2.9]{jmm_mwj:2010}.

It can be more convenient to work with conjugacy classes of subgroups rather than the subgroups themselves.
Let $[\cat SG{p+\mathrm{rad}}]$ be the set of conjugcacy classes $[K]$ of $p$-radical subgroups $K$ of $G$.
Since  $|\cat OG{}(H,K)| =| (\qt KG)^H|$ is the
  mark of $H$ on the transitive right $G$-set $\qt KG$, 
the square matrix $[\cat OG{p+\mathrm{rad}}]$ with entries
\begin{equation}
  \label{eq:TOM}
  [\cat OG{p+\mathrm{rad}}]([H],[K]) = | \cat OG{}(H,K)|, \qquad
  [H],[K] \in[\cat SG{p+\mathrm{rad}}]
\end{equation}
is Burnside's {\em table of
  marks\/} \cite{burnside55}  for the  $p$-radical subgroup classes. It is easy to see that the category
$\cat OG{p+\mathrm{rad}}$ and the table of marks $[\cat OG{p+\mathrm{rad}}]$ have the same \Euc\  
  \cite[\S2.4]{jmm_mwj:2010}.

\begin{prop}\label{prop:GpEuc}
  Let $G$ be a finite group and $p$ a prime number.
  \begin{enumerate}
  \item \label{item:GpEuc1}$\displaystyle
   \sum\limits_{K \in \cat SG{p+\mathrm{rad}}}
  -\rchi(\cat S{\cat OG{}(K)}{p+*})|K| =     |G_p|
  $ 
\item 
  \label{item:GpEuc2}
    $\sum\limits_{K \in \cat SG{p+\mathrm{rad}}}
  -\rchi(\cat S{\cat OG{}(K)}{p+*}) = 1$
\item \label{item:GpEuc3}
    $\sum\limits_{[K] \in [\cat SG{p+\mathrm{rad}}]} -\rchi(\cat S{\cat OG{}(K)}{p+*}) \frac{|\cat OG{}(H,K)|}{|\cat OG{}(K)|}
  = 1$ for any $p$-radical subgroup $H$ 
  \end{enumerate}
   \end{prop}
\begin{proof}
  Lemma~\ref{lemma:Gp} combined with \cite[Theorem 1.3.(4)]{jmm_mwj:2010} show that
  \begin{equation*}
    |G_p| = |G|\chi(\cat OG{p+\mathrm{rad}}) = \sum_{K \in \cat S{G}{p+\mathrm{rad}}} -\rchi(\cat S{\cat
      OG{}(K)}{p+*}) |K| 
  \end{equation*}
  where the sum ranges over all $p$-radical subgroups $K$ of $G$. This
  proves \eqref{item:GpEuc1}. Item \eqref{item:GpEuc2}  simply expresses that
  $\cat SG{p+\mathrm{rad}}$, with $O_p(G)$ as its least element
  \cite[Proposition~6.3]{gm:2012}, has \Euc\ equal to $1$.

  The \we ,
  $k^\bullet_{[\cat O{}{}]} \colon [\cat SG{p+\mathrm{rad}}] \to \Q$,
  for the table of marks of the $p$-radical subgroup classes
  \eqref{eq:TOM} satisfies
  \begin{equation*}
    \sum_{[K] \in [\cat SG{p+\mathrm{rad}}]} |\cat OG{}(H,K)| k^K_{[\cat O{}{}]}   = 1
  \end{equation*}
  for all $p$-radical subgroups $H$. The \we s, $k^\bullet_{\cat O{}{}}$ and
  $k^\bullet_{\cat S{}{}}$, for $\cat OG{p+\mathrm{rad}}$ and
  $\cat SG{p+\mathrm{rad}}$ are \cite[Proposition
  2.14]{jmm_mwj:2010}
\begin{equation*}
  k_{\cat O{}{}}^K = \frac{|N_G(K)|}{|G|}k_{[\cat O{}{}]}^{[K]}, \qquad
    k_{\cat S{}{}}^K = \frac{|G|}{|K|} k_{\cat O{}{}}^K = \frac{|G|}{|K|} \frac{|N_G(K)|}{|G|} k_{[\cat O{}{}]}^K =
    |\cat OG{}(K)| k_{[\cat O{}{}]}^K 
  \end{equation*}
  and therefore
    \begin{equation*}
    \sum_{[K] \in [\cat SG{p+\mathrm{rad}}]} \frac{|\cat OG{}(H,K)|}{|\cat OG{}(K)|} k^K_{\cat S{}{}}   = 1
  \end{equation*}
  which is the third item.
\end{proof}

Proposition~\ref{prop:GpEuc}.\eqref{item:GpEuc1} expresses that $|G_p|
= |G| \chi(\cat OG{p+\mathrm{rad}}) = |G| \chi([\cat
OG{p+\mathrm{rad}}])$ can be computed from the table of marks for the
$p$-radical subgroups \eqref{eq:TOM}.

The content of Proposition~\ref{prop:GpEuc}.\eqref{item:GpEuc3} is
that the  vector $(k_{\cat S{}{}}^K)_{K \in [\cat SG{p+\mathrm{rad}}]}$ is a \we\
for $[[\cat OG{p+\mathrm{rad}}]]$, {\em the modified table of
  marks},  defined to be the square matrix 
with entries  
\begin{equation}
  \label{eq:modTOM}
  [[\cat OG{p+\mathrm{rad}}]]([H],[K]) = \frac{| \cat OG{}(H,K)|}{|\cat OG{}(K)|}, \qquad
  [H],[K] \in[\cat SG{p+\mathrm{rad}}]
\end{equation}
In other words, $k^K_{\cat S{}{}} = k^K_{[[\cat O{}{}]]}$ for all
$p$-radical subgroups $K$ where $k^\bullet_{[[\cat O{}{}]]}$ is the
\we\ for $[[\cat OG{p+\mathrm{rad}}]]$.


The normalizer $N_G(K)$ acts on the transporter set
$N_G(H,K) =\{ g \in G \mid H^g \subseteq K\}$ and the orbit set
corresponds bijectively via the map $g \to K^{g^{-1}}$ to the set
$\{ L \in [K] \mid H \supseteq L\}$ of conjugates of $K$ containing
$H$. Thus the modified mark
\begin{equation*}
  \frac{|\cat OG{}(H,K)|}{|\cat OG{}(K)|} = |N_G(H,K)/N_G(K)| = |\{ L \in [K] \mid H \supseteq L\}|
\end{equation*}
is the number of $H$-supergroups conjugate to $K$.

\begin{exmp}[$G=\Sigma_4$, $p=2$]  
  The $2$-radical subgroup classes of the symmetric group $\Sigma_4$ are the
  \syl 2 $D_8$  and $O_2(\Sigma_4)=C_2 \times C_2$ of
  order $4$. The table of marks \eqref{eq:TOM} and the modified table of marks \eqref{eq:modTOM} for the
  $2$-radical subgroup classes in $\Sigma_4$ are
  \begin{equation*}
    [\cat O{\Sigma_4}{2+\mathrm{rad}}] =
    \begin{pmatrix}
      1 & 0 \\ 3 & 6
    \end{pmatrix}, \qquad
    [[\cat O{\Sigma_4}{2+\mathrm{rad}}]] =
    \begin{pmatrix}
      1 & 0 \\ 3 & 1
    \end{pmatrix}
  \end{equation*}
  The \we\ for the table of marks is $k_{[\cat O{}{}]} = (1,-1/3)$ and
  $|\Sigma_4| \chi([\cat O{\Sigma_4}{2+\mathrm{rad}}]) = 24(1-1/3) =
  16$ is the number of $2$-singular elements in $\Sigma_4$ in
  agreement with Proposition~\ref{prop:GpEuc}.\eqref{item:GpEuc1}. The
  modified table of marks has \we\ $k_{[[\cat O{}{}]]} = (1,-2)$ which
  by Proposition~\ref{prop:GpEuc}.\eqref{item:GpEuc3} means that
  $k^{D_8}_{\cat S{}{}} = 1$ and
  $k^{C_2 \times C_2}_{\cat S{}{}} = -2$. Since the subgroups $D_8$,
  $C_2 \times C_2$ have lengths $3$, $1$, the \Euc\ of the Brown poset
  $\cat S{\Sigma_4}{2+\mathrm{rad}}$ is
  $\chi(\cat S{\Sigma_4}{2+\mathrm{rad}}) = 3 \cdot 1 + 1 \cdot (-2)
  =1$ in agreement with
  Proposition~\ref{prop:GpEuc}.\eqref{item:GpEuc2}.
\end{exmp}

\section{The theorems of Frobenius and Brown for finite groups are equivalent}
\label{sec:brown}

The following theorem was proved by K.S. Brown \cite{brown75}
(and reproved by Quillen \cite[Corollary
4.2]{quillen78}, Webb \cite[Theorem 8.1]{webb87} and others).

\begin{thm}[Brown $1975$]  \label{thm:brown}
  $|G|_p \mid \rchi(\cat SG{p+*})$
\end{thm}

It was observed in \cite{brown-thevenaz88, HIO89} that M\"obius
functions link the theorems of Frobenius and Brown.  We here note
that also Proposition~\ref{prop:GpEuc}.\eqref{item:GpEuc1} connects the
two theorems.

\begin{prop}\label{thm:frobeniusbrown}
  Theorems~\ref{thm:frobenius} and \ref{thm:brown} are
  equivalent  given  Proposition~\ref{prop:GpEuc}.\eqref{item:GpEuc1}. 
\end{prop}
\begin{proof}
  Proposition~\ref{prop:GpEuc}.\eqref{item:GpEuc1} may be rewritten on
  the form
\begin{equation}
  \label{eq:chiOG}
    |G_p| + 
    \rchi(\cat SG{p+*}) + \sum_{[H] \neq 1} 
   \frac{\rchi(\cat S{\cat OG{}(H)}{p+*})}{|\cat OG{}(H)|_p}
   \frac{|G|}{|\cat OG{}(H)|_{p'}} = 0
 \end{equation}
where we have isolated the contribution from the trivial
  subgroup and the sum is over classes of non-trivial $p$-radical
  subgroups of $G$.

  Assume first that Theorem~\ref{thm:frobenius} holds.  In
  Equation~\eqref{eq:chiOG}, we may assume that
  \begin{itemize}
  \item $\rchi(\cat S{\cat OG{}(H)}{p+*})/ |\cat OG{}(H)|_p$ is an integer
  when $H$ is nontrivial (as part of an inductional argument)
\item $|G|/|\cat OG{}(H)|_{p'}$ is an integer divisible by $|G|_p$ (as
  $|\cat OG{}(H)|$ divides $|G|$)
  \end{itemize} 
  Thus every term in the sum is divisible by $|G|_p$ and so is $|G_p|$
  by assumption.  We conclude that $\rchi(\cat SG{p+*})$ is divisible
  by $|G|_p$ and we have arrived at Theorem~\ref{thm:brown}.

 Assume next that Theorem~\ref{thm:brown} holds.  In
  Equation~\eqref{eq:chiOG}
  \begin{itemize}
  \item $|G|/|\cat OG{}(H)|_{p'}$ is an integer divisible by $|G|_p$
\item $\rchi(\cat S{\cat OG{}(H)}{p+*})/|\cat OG{}(H)|_p$ is an integer
\item $\rchi(\cat SG{p+*})$ is divisible by $|G|_p$
\end{itemize}  
and thus $|G|_p$ divides for $|G_p|$. This is 
Theorem~\ref{thm:frobenius}.
\end{proof}

\section{The theorems of Solomon and Steinberg 
  for finite groups of Lie type are equivalent}
\label{sec:finite-groups-lie}

Let $\Sigma$ be a reduced and crystallographic root system with
fundamental and positive roots $\Pi, \Sigma^+ \subseteq \Sigma$
\cite[Definition 1.8.1] {GLSIII}. Suppose
$\overline{K}(\Sigma)$ is a semisimple
$\overline{\F}_p$-algebraic group with root system $\Sigma$
\cite[Theorem 1.10.4]{GLSIII} equipped with a (standard form) Steinberg endo\m\ $\sigma$
\cite[Definition~1.15.(b), Remarks~2.2.5.(e)]{GLSIII}.  Assuming $\Sigma$ to be also
irreducible \cite[Definition 1.8.4]{GLSIII}, let
$K=O^{p'}C_{\overline{K}(\Sigma)}(\sigma)$ be the finite group in
$\Lie(p)$ with $\sigma$-setup $(\overline{K}(\Sigma),\sigma)$
\cite[Definition~2.2.2]{GLSIII}.

The number of $p$-singular elements in $K$ was determined by Steinberg
\cite[15.2]{steinberg68}. 

\begin{thm}[Steinberg $1968$]\label{thm:steinberg}
  $|K_p| = |K|_p^2$
\end{thm}

The surjections $\Sigma \to \widetilde{\Sigma} \to \wh{\Sigma}$ of
\cite[(2.3.1)]{GLSIII} induce surjections
$\Pi \to \widetilde{\Pi} \to \wh{\Pi}$ of sets.  Here,
$\widetilde{\Sigma}$ is the twisted root system of $K$ \cite[p
41]{GLSIII}, and $\wh{\Sigma} =\widetilde{\Sigma}/\!\!\sim$ is the set of equivalence classes of
twisted roots pointing in the same direction. If $K$ is an untwisted group of Lie type
\cite[Definition 2.2.4]{GLSIII}, $\Sigma = \widetilde{\Sigma} = \wh\Sigma$.



For every subset $J \subseteq \wh{\Pi}$ we have associated subgroups
$P_J,U_J,L_J \subseteq K$ such that $U_J=O_pP_J$, $P_J=N_K(U_J)$ and
$P_J = U_J \rtimes L_J$ \cite[Theorem~2.6.5]{GLSIII}. The $P_J$ are
parabolic subgroups, the $U_J$ are unipotent $p$-radical subgroups and the $L_J$
are Levi complements \cite[Definition~2.6.4,
Definition~2.6.6]{GLSIII}.  The extreme cases where $J=\emptyset,\wh\Pi$
are, $P_\emptyset = U_\emptyset \rtimes L_\emptyset$, where
$P_\emptyset = B$ is a Borel subgroup of $K$, $U_\emptyset=U$ a Sylow
$p$-subgroup \cite[p 41, Theorems 2.3.4, 2.3.7]{GLSIII} and
$L_\emptyset = H$ is a maximal torus or Cartan subgroup \cite[Theorem
2.4.7, Definition 2.4.12]{GLSIII}, and
$P_{\wh \Pi} = K = L_{\wh \Pi}$, $U_{\wh\Pi} = 1$.

The following \pol\ identity dates back to L. Solomon
\cite[Corollary~1.1]{solomon66} \cite[Theorem~9.4.5, \S
14]{carter:lie}.

\begin{thm}[Solomon $1966$]\label{thm:solomon}
    $\sum\limits_{J \subseteq \widehat{\Pi}}
    (-1)^{|J|}|P_{\widehat{\Pi}} : P_J| = |K|_p$
  \end{thm}

  Solomon's theorem, essentially a statement about
  reflection groups, generalizes to the following two identities that
  are \Mb\ inverses to each other.

  \begin{cor}\label{cor:solomon}
    For any subset $I$ of $\wh\Pi$,  
  $
    \sum\limits_{I \supseteq J} (-1)^{|J|} |P_I : P_J| = |L_I|_p$ and
    $\sum\limits_{I \supseteq J} (-1)^{|J|} |P_I : P_J| |L_J|_p = 1$.
  \end{cor}

  We use a consequence of the Borel--Tits theorem \cite[Theorem~3.1.3]{GLSIII} to
  determine the modified table of marks for the $p$-radical subgroups
  of $K$.
  
\begin{lemma}\label{lemma:PIPJ}
  The modified table of marks \eqref{eq:modTOM} for the $p$-radical
  subgroups, $U_I$, $I \subseteq \wh\Pi$, of $K$ has entries
  \begin{equation*}
   [[\cat OK{p+\mathrm{rad}}]](U_I,U_J) =
   \begin{cases}
    |P_I : P_J|    & I \supseteq J \\ 0 & \text{otherwise}     
   \end{cases}
  \end{equation*}
  for all subsets $I,J \subseteq \wh\Pi$.
\end{lemma}
\begin{proof}
  By \cite[Corollary~3.16]{GLSIII}, assisted by
  \cite[Theorem~2.6.7]{GLSIII} to show that any parabolic subgroup
  containing $U$ contains $B$, the set
  $\{U_I \mid I \subseteq \wh\Pi\}$ is the Alperin--Goldschmidt
  conjugation family \cite[Theorem~16.1]{GLSII} controlling fusion in
  $U$.
  It follows that if $g \in K$ conjugates $U_I$ into some $U_J$
  then $U_I^g = U_I$ so that $I \supseteq J$ and $g \in N_K(U_I)=P_I$
  \cite[Theorem~2.6.5]{GLSIII}. This shows that the transporter set
  $N_K(U_I,U_J)$ equals $N_K(U_I)$ in case $I$ contains $J$ and is
  empty otherwise.
\end{proof}

This leads to a  weak version of the
Solomon--Tits theorem \cite[Corollary~7.3]{curtis-lehrer-tits80}. 

\begin{cor}\label{cor:chiLJ}
  $-\rchi(\cat S{L_I}{p+*}) = (-1)^{|I|} |L_I|_p$ for all $I \subseteq \wh\Pi$.
\end{cor}
\begin{proof}
  This follows immediately from
  Proposition~\ref{prop:GpEuc}.\eqref{item:GpEuc3} and the second
  identity of Corollary~\ref{cor:solomon} as the coefficients
  $|P_I : P_J|$ are the the modified marks by Lemma~\ref{lemma:PIPJ}.
\end{proof}

\begin{exmp}
  The group $\GL {}3{2} \in \Lie(2)$ has fundamental roots
  $\Pi=\{\alpha_1,\alpha_2\}$. Its parabolic subgroups, $P_\emptyset$,
  $P_{\{\alpha_1\}}$, $P_{\{\alpha_2\}}$, $P_{\Pi}$, have orders
  $8, 8 \cdot 3, 8 \cdot 3, 8 \cdot 21$ and Levi complements, $1$,
  $\GL{}22$, $\GL{}22$, $\GL{}32$, with \syl 2\ orders
  $1,2,2,8$. The signed vector $(1,-2,-2,8)$ of
  Corollary~\ref{cor:chiLJ} is indeed the \we\ for
  the modified table of marks of Lemma~\ref{lemma:PIPJ} as
  \begin{equation*}
    \begin{pmatrix}
      1 & 0 & 0 & 0 \\
      3 & 1 & 0 & 0 \\
      3 & 0 & 1 & 0 \\
      21 & 7 & 7 & 1 
    \end{pmatrix}
    \begin{pmatrix}
      1 \\ -2 \\ -2 \\ 8
    \end{pmatrix} =
    \begin{pmatrix}
      1 \\ 1 \\ 1 \\ 1
    \end{pmatrix}
  \end{equation*}
  By Proposition~\ref{prop:GpEuc}.\eqref{item:GpEuc1}, this \we\ for
  the modified table of marks records the negative reduced \Euc s,
  $-\rchi(\cat S{L_I}{2+*})$, of the Brown posets for the Levi
  complements.
\end{exmp}

We are now prepared to prove a version of Steinberg's
theorem valid for all parabolic subgroups of $K$.

\begin{thm}\label{thm:gensteinberg} 
  $|P_p| |O_p(P)| = |K|_p^2$ for any parabolic subgroup $P$ of $K$.
\end{thm}
\begin{proof}
  There is always a bjiection between the $p$-radical subgroups of a
  finite group $G$ and those of $G/O_p(G)$
  \cite[Proposition~6.3]{gm:2012}. In particular, $\{U_J \mid I
  \supseteq J\}$ is a complete set of representatives for the
  $p$-radical subgroups of $P_I$ corresponding to the $p$-radical
  subgroup classes of $L_I$.
  Obviously,
  $N_{P_I}(U_J) = P_I \cap N_K(U_J) = P_I \cap P_J = P_J =N_K(U_J)$,
  $\cat O{P_I}{}(U_J) = P_J/U_J = L_J = \cat OK{}(U_J)$ and
  $|P_I : N_{P_I}(U_J)| = |P_I : P_J|$ is the length of $U_J$ in $P_I$
  and $K$. By Proposition~\ref{prop:GpEuc}.\eqref{item:GpEuc1}, the
  number of $p$-singular elements in $P_I$ is
  \begin{multline*}
    |(P_I)_p| =
    \sum_{I \supseteq J} -\rchi(\cat S{L_J}{p+*}) |P_I : P_J| |U_J| =
    \sum_{I \supseteq J} (-1)^{|J|} |P_I : P_J| |L_J|_p |U_J| =
    |K|_p\sum_{I \supseteq J} (-1)^{|J|} |P_I : P_J| \\ \stackrel{\text{Cor~\ref{cor:solomon}}}{=}
    |K|_p |L_I|_p =
    |K|_p |P_I : U_I|_p =
    |K|_p^2/|U_I|
  \end{multline*}
This finishes the proof as $U_I=O_p(P_I)$.
\end{proof}

Theorems~\ref{thm:steinberg} and \ref{thm:gensteinberg} apply e.g.\ to
$\Omega_{2m}^\pm(\F_q)$, $P\Omega_{2m}^\pm(\F_q)$,
$\Spineven{\pm}{2m}{\F_q}$ for all prime powers $q$, to
$\SOeven{\pm}{2m}{\F_{q}}$ for odd $q$,
but not to
$\SOeven{\pm}{2m}{\F_{2^e}}$ \cite[\S2.7]{GLSIII}. ($\SOeven{-}4{\F_2}
\cong \Sigma_5$ of order $8 \cdot 15$ contains $56<8^2$ $2$-singular elements.)
They also apply to $\Omega_{2m+1}(\F_q)$, $\SOodd{2m+1}{\F_q}$ and
$\Spinodd{2m+1}{\F_q}$ for all $q$.

The $q$-bracket of the natural number $d$ is the \pol\
$[d](q) = q^{d-1}+\cdots+q+1 \in \Z[q]$ of degree $d-1$ with value
$[d](1)=d$ at $q=1$. For a reflection group $W$, put
$W(q)= \prod_d [d](q)$ where the product is over the degrees $d$ of
the basic \pol\ invariants \cite[Proposition 10.2.5]{carter:lie}.  The
two identities of Corollary~\ref{cor:solomon} with $I=\Pi$ for the
Chevalley (untwisted) groups associated to $\overline{K}(\Sigma)$ with
Weyl group $W$,
\begin{equation}\label{eq:witt}
 \sum_{J \subseteq \Pi} (-1)^J \frac{W_{\Pi}(q)}{W_J(q)} =
 q^{|\Sigma^+|}, \qquad
 \sum_{J \subseteq \Pi} (-1)^J \frac{W_{\Pi}(q)}{W_J(q)} 
  q^{|\Sigma^+_J|} = 1
\end{equation}
are two $q$-analogs of Witt's identity \cite{witt1941}
\cite[(5)]{solomon66}
$ \sum\limits_{J \subseteq \Pi} (-1)^J |W_{\Pi} : W_J|=1$.

Let
$\mathrm{OP}(m)=\{(m_1,\ldots,m_k) \mid k \geq 1, m_i \geq 1,\sum
m_i=m\}$ denote the set of all the $2^{m-1}$ ordered partitions of $m$
\cite[p 14]{stanley97}.

\begin{exmp}
  Subsystems of the root systems $A_{m-1}$ or $B_{m-1}$ are
  indexed  by $\mathrm{OP}(m)$ via the bijection taking
  $(m_1,\ldots,m_k) \in \mathrm{OP}(m)$ to $A_{m_1-1} \times \cdots
  \times A_{m_k-1}$ or $A_{m_1-1} \times \cdots
  \times A_{m_{k-1}-1} \times B_{m_k-1}$ (where $A_0$ is the empty root system
  and $B_0=A_0$, $B_1=A_1$).
  The incarnations of equations \eqref{eq:witt} for the Chevalley
  groups $\SL{+}{m}{\F_q}$ and $\SOodd{2m-1}{\F_q}$  of rank $m-1$ with root systems
  $\Sigma = A_{m-1},B_{m-1}$ are the \pol\ identities 
  \begin{gather*}
    \sum
  (-1)^k \binom{[m](q)}{[m_1](q),\cdots,[m_k](q)} =(-1)^m q^{\binom
    m2}, \qquad
   \sum
  (-1)^k  \binom{[m](q)}{[m_1](q),\cdots,[m_k](q)}q^{\sum \binom{m_i}2} = (-1)^m \\
    \sum  \frac{(-1)^k\prod\limits_{d=m_k}^{m-1}[2d](q)}{[m_1]!(q)\cdots[m_{k-1}]!(q)} =(-1)^m q^{(m-1)^2}, \qquad
   \sum 
  \frac{(-1)^k\prod\limits_{d=m_k}^{m-1}[2d](q)}{[m_1]!(q)\cdots
    [m_{k-1}]!(q)} q^{\sum_{i=1}^{m_{k-1}} \binom{m_i}{2}+(m_k-1)^2}
  =(-1)^m  
\end{gather*}
The sums are indexed by all $(m_1,\ldots,m_k) \in \mathrm{OP}(m)$ and the identities for
$A_{m-1}$ use Gaussian multinomial coefficients
\cite[\S1.7]{stanley97}.
\end{exmp}

\begin{exmp}[$\SL{-}{m}{\F_q}$]  
  The two identities of Corollary~\ref{cor:solomon} with $I=\wh\Pi$
  for the Steinberg group $\SL{-}{2m}{\F_q}$ of rank $2m-1$ and
  twisted rank $m$ are
  \begin{gather*}  
 \sum
  \frac{(-1)^{k} \prod\limits_{d=1}^{2m}[d]((-1)^dq)}{[m_1]!(q^2)
    \cdots [m_k]!(q^2)} -
  \sum
  \frac{(-1)^{k} \prod\limits_{d=2m_k+2}^{2m}[d]((-1)^dq)}{[m_1]!(q^2)
    \cdots [m_{k-1}]!(q^2) } = (-1)^m q^{\binom{2m}{2}} \\
    \sum
  \frac{(-1)^{k} \prod\limits_{d=1}^{2m}[d]((-1)^dq)}{[m_1]!(q^2)
    \cdots [m_k]!(q^2)} q^{\sum \binom{m_i}{2}} -
  \sum
  \frac{(-1)^{k} \prod\limits_{d=2m_k+2}^{2m}[d]((-1)^dq)}{[m_1]!(q^2)
    \cdots [m_{k-1}]!(q^2) } q^{\sum \binom{m_i}{2}} = (-1)^m
\end{gather*}
and for the Steinberg group $\SL{-}{2m+1}{\F_q}$ of rank $2m$ and
twisted rank $m$ they are
\begin{gather*}  
 \sum
  \frac{(-1)^{k} \prod\limits_{d=1}^{2m+1}[d]((-1)^dq)}{[m_1]!(q^2)
    \cdots [m_k]!(q^2)} -
  \sum
  \frac{(-1)^{k} \prod\limits_{d=2m_k+2}^{2m+1}[d]((-1)^dq)}{[m_1]!(q^2)
    \cdots [m_{k-1}]!(q^2) } = (-1)^m q^{\binom{2m+1}{2}} \\
    \sum
  \frac{(-1)^{k} \prod\limits_{d=1}^{2m+1}[d]((-1)^dq)}{[m_1]!(q^2)
    \cdots [m_k]!(q^2)} q^{\sum \binom{m_i}{2}} -
  \sum
  \frac{(-1)^{k} \prod\limits_{d=2m_k+2}^{2m+1}[d]((-1)^dq)}{[m_1]!(q^2)
    \cdots [m_{k-1}]!(q^2) } q^{\sum \binom{m_i}{2}} = (-1)^m
\end{gather*}
where the sums run over all  $(m_1,\ldots,m_k) \in \mathrm{OP}(m)$.
These identities are obtained by analyzing the $C_2$-subsystems of the
$C_2$-root system $A_{m-1}$ \cite[13.3.8]{carter:lie}.
Write $S(A_{m-1})$ for the multiset of all
  $C_2$-subsystems of $A_{m-1}$. One subsystem of $A_{2m-1}$ is $a_{2m-1}$
  defined to be the free part of $A_{2m-1}$, i.e.\@ the subsystem
  obtained by deleting the middle root $\alpha_m$. The fundamental
  roots of the $C_2$-root systems $a_1,a_3,a_5,a_7$ are
  
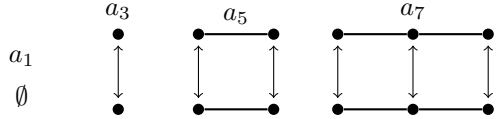
\begin{equation*}
    \begin{tikzpicture}
      \node at (0,0) {};
      \node[]  at (0,-1) {$\emptyset$};
      \node at (0,-.5) {$a_1$};
    \end{tikzpicture} \qquad
    \begin{tikzpicture}
      \node[circle, fill=black, inner sep=1.5pt] (1)   {};
      \node[circle, fill=black, inner sep=1.5pt, label={$a_3$}] (2) [above of=1] {};
      \draw[<->, shorten >=2pt, shorten <=2pt ] (1) to  (2); 
    \end{tikzpicture} \qquad
    \begin{tikzpicture}
      \node[circle, fill=black, inner sep=1.5pt] (1) {};
      \node[circle, fill=black, inner sep=1.5pt] (2) [right of=1] {};
      \node[circle, fill=black, inner sep=1.5pt] (3) [above of=2] {};
      \node[circle, fill=black, inner sep=1.5pt] (4) [above of=1] {};
      \draw[<->, shorten >=2pt, shorten <=2pt] (1) to (4);
      \draw[<->, shorten >=2pt, shorten <=2pt] (2) to (3);
      \draw[-,thick] (1)--(2);
      \draw[-,thick] (3)-- node[above] {$a_5$}  (4);
    \end{tikzpicture} \qquad
    \begin{tikzpicture}
      \node[circle, fill=black, inner sep=1.5pt] (1) {};
      \node[circle, fill=black, inner sep=1.5pt] (2) [right of=1] {};
      \node[circle, fill=black, inner sep=1.5pt] (3) [right of=2] {};
      \node[circle, fill=black, inner sep=1.5pt] (4) [above of=3] {};
      \node[circle, fill=black, inner sep=1.5pt, label={$a_7$}] (5) [above of=2] {};
      \node[circle, fill=black, inner sep=1.5pt] (6) [above of=1] {};
      \draw[<->, shorten >=2pt, shorten <=2pt] (1) to (6);
      \draw[<->, shorten >=2pt, shorten <=2pt] (2) to  (5);
      \draw[<->, shorten >=2pt, shorten <=2pt] (3) to  (4);
      \draw[-,thick] (1)--(2);
      \draw[-,thick] (2)--(3);
      \draw[-,thick] (4)--(5);
      \draw[-,thick] (5)--(6);
    \end{tikzpicture}
  \end{equation*}
  The first multisets of subsystems  are $S(A_1)=\{a_1,A_1\}$,
  $S(A_2)=\{a_1,A_2\}$,
  $S(A_3)=\{a_1,A_1,a_3,A_3\} = a_1 \times S(A_1) \cup
  \{a_3,A_3\}$, $S(A_4)=\{a_1,A_2,a_3,A_4\} = a_1 \times S(A_2) \cup
  \{a_3,A_4\}$.  In general, the $2^m$ subsystems of $A_{2m-1}$ and
  $A_{2m}$, 
  $m\geq 2$, are the multisets
  \begin{align*}
    S(A_{2m-1}) &= a_1 \times S(A_{2m-3}) \cup \cdots \cup a_{2i-1}
    \times S(A_{2(m-i)-1}) \cup \cdots \cup a_{2m-3} \times S(A_1)
      \cup \{a_{2m-1},A_{2m-1}\} \\
    S(A_{2m}) &= a_1 \times S(A_{2m-2}) \cup \cdots \cup a_{2i-1}
      \times S(A_{2(m-i)}) \cup \cdots \cup a_{2m-3} \times S(A_2)
      \cup \{a_{2m-1},A_{2m}\}
  \end{align*}
  For each subsystem $a$ of $A_m$, let $P(a)(q) =|P : B|\in \Z[q]$ be
  the index of the Borel subgroup $B$ in the parabolic subgroup of
  $\SL{-}{m+1}{\F_q}$ corresponding to $a$. In particular, $P(A_m)(q)$
  and $P(a_{2m-1})(q)$ are the \pol s
\begin{equation*}
  P(A_m)(q)= \prod_{1 \leq d \leq m+1} [d]((-1)^dq), \quad
  P(a_{2m-1}) = \prod_{1 \leq d \leq m} [d](q^2) = [m]!(q^2)
  \qquad m \geq 1
\end{equation*}
of degrees $\binom{m+1}{2}$ and $\binom{m}{2}$.
Consider the multiset of signed \pol s associated to all subsystems of $A_m$ 
\begin{equation*}
  P(S(A_m)) =\{ (-1)^{|\Pi(a)/C_2|} P(a)(q) \mid a \in S(A_m) \}
\end{equation*}
where $\Pi(a)$ is the set of fundamental roots and $\Pi(a)/C_2$ the
orbit set.  Then $P(S(A_1)) = \{1,-P(A_1)\}$,
$P(S(A_2)) = \{1,-P(A_2)\}$ and one may now determine the multisets of
\pol s for all the $C_2$-root systems $A_{2m-1}$ and $A_{2m}$,
$m \geq 2$. This leads to the above \pol\ identities.
\end{exmp}

Steinberg's theorem applies in the equicharacteristic case and does
not hold in the cross-characteristic case.  However, it is known that
the number of $p$-singular classes in $\GL {}nq$, $p \nmid q$, is
\begin{equation*}  
  | \GL {}nq_p / \GL{}nq | = \frac{1}{n!}
  \sum_{\lambda \vdash n} T(\lambda) \prod_{b \in \lambda} (q^b-1)_p
\end{equation*}
where $\lambda$ ranges over all partitions of $n$ and $T(\lambda)$ is
the number of permutations of cycle type $\lambda$ in $\Sigma_n$ \cite[Corollary~4.22]{jmm:eulergl+}.
The number of $p$-singular classes, but not the number of $p$-singular
elements, in $\GL{}nq$, $p \nmid q$, depends only on the $p$-fusion system.

\section*{Acknowledgments}
\label{Acknowledgments}
I thank the Department of Mathematics at the Universitat Aut\`onoma
Barcelona for kind hospitality and for the opportunity to speak in the
Topology Seminar. 
Lemma~\ref{lemma:Gp} was pointed out by Sune Precht Ree.  Warm thanks also go to Bob Oliver for his
interest and Justin Lynd for referring me to Steinberg's paper
\cite{steinberg68}. This note would not have been written without
their help and encouragement.


\def\cprime{$'$} \def\cprime{$'$} \def\cprime{$'$} \def\cprime{$'$}
  \def\cprime{$'$}
\providecommand{\bysame}{\leavevmode\hbox to3em{\hrulefill}\thinspace}
\providecommand{\MR}{\relax\ifhmode\unskip\space\fi MR }
\providecommand{\MRhref}[2]{%
  \href{http://www.ams.org/mathscinet-getitem?mr=#1}{#2}
}
\providecommand{\href}[2]{#2}

\end{document}